\documentclass[article,11pt,leqno]{article}
\usepackage{articlemacro}

\newcommand{\DM}{\mathrm{DM}}

\newcommand{\PreStk}{\mathrm{PreStk}}

\newcommand{\Sht}{\mathrm{Sht}}
\newcommand{\GSht}{\mathrm{Sht}_{G}}
\newcommand{\nSht}[1]{#1\textup{-}\Sht}

\newcommand{\Disp}{\mathrm{Disp}}
\newcommand{\quot}[1]{\left[#1\right]}

\newcommand{\Zip}[1]{#1\textup{-Zip}}

\date{}
\title{Truncations and the Motive of the Stack of Local $G$-Shtukas}
\author{Can Yaylali\footnote{\href{mailto:can.yaylali@uni-due.de}{can.yaylali@uni-due.de}, University of Duisburg-Essen, Essen, Germany\\ ORCID: 0000-0003-0150-3622}}

\begin{document}
\maketitle

\begin{abstract}
	We compute the rational motive of the stack of $N$-truncated local $G$-shtukas, for a split reductive group $G$, representing compactly supported cohomology in terms of the motive of the stack of $G$-zips. As an application, we prove that the motive of the stack of (untruncated) local $G$-shtukas of type $\leq \mu$ admits a similar description, assuming that the rank of the maximal torus is $1$ and the length of $\mu\leq 1$. This result makes use of the theoretical background on the $*$- and $!$-adjunction for pro-systems of algebraic objects established by the author.
\end{abstract}

\thispagestyle{empty}
\tableofcontents

\section*{Introduction}
Shtukas occupy a central position in the Langlands program for function fields: their moduli spaces provide geometric realizations of automorphic forms and were used by L. Lafforgue to establish the global Langlands correspondence for $\GL_n$.  Their \emph{local} analogues were introduced by Hartl–Viehmann \cite{HV} play an important role in the proof of the tame categorical local Langlands correspondence  by Xinwen Zhu \cite{Zhu}.  In a parallel but complementary direction, Richarz–Scholbach have established a motivic form of the geometric Satake equivalence with $\QQ$-coefficients, providing a motivic incarnation of one of the key structures in the Langlands program \cite{RS1}. 

\medskip

In this article we study the \emph{rational motive} of the stack of local $G$-shtukas, where $G$ is a split reductive group. Our goal is to describe its motivic cohomology with compact support in purely motivic terms and, in particular, to relate it to the finite-type stacks of $G$-zips that appear in its stratification. The resulting comparison shows that
the cohomological complexity of the stack of local $G$-shtukas is completely captured by its truncated models.

\medskip

Local $G$-shtukas, as introduced by Hartl–Viehmann, are defined over $\FF_p$ and may be viewed as function-field analogues of $p$-divisible groups.  Their moduli stacks are quotients of ind-schemes by pro-algebraic groups and highly infinite-dimensional, but recent work of Viehmann–Wedhorn has shown that they can be  pproximated by truncated versions whose strata are described by stacks of $G$-zips \cite{VW}.  Building on this insight, we develop a motivic framework in which these approximations can be compared directly on the level of rational motives. 

\medskip 

The key input is a localization sequence for motives on quotients of ind-schemes by pro-algebraic groups with split pro-unipotent kernel. Using recent advances in the theory of motives on ind- and pro-Artin stacks, we establish the existence of a $\ast$-adjunction and compact-object-preserving pushforward for such quotients, together with a localization  formula.  This technical step allows us to compute motivic cohomology with compact support in families of local shtukas.

\begin{intro-theorem}[\protect{\ref{thm.main}}]
\label{intro-thm}
	Let $G$ be a split reductive group and $T$ a split maximal torus of rank $1$ contained in $G$. Let $\mu\in X_{*}(T)_{+}$ be a dominant quasi-minuscule cocharacter. Let $\GSht^{\leq\mu}$ denote the closure\footnote{See Section \ref{sec-GSht} for the notation.} of the stack of local $G$-shtukas of type $\mu$. Then 
	\[
		M^{c}(\GSht^{\leq\mu}) \simeq \bigoplus_{\mu'\leq \mu} M^{c}(\Zip{G}^{\mu'})
	\]
	in $\DM(\FF_{p})$.
	In particular, it is a Tate motive.
\end{intro-theorem}

 This expresses the Borel-Moore homology of the (infinite-dimensional) stack of local $G$-shtukas of type $\leq \mu$ in terms of the finite-type stacks of $G$-zips arising through its stratification.  In other words, all homological information of $\Sht_G^{\leq \mu}$ can already be recovered from its truncated models. Consequently, all information about the motivic cohomology of  $\Sht_G^{\mu}$ is already encoded in the finite-dimensional stacks of $G$-zips, whose motives and Chow rings are  explicitly known \cite{Brokemper,WZ,MotGZip}. 
 
 Note that we have very strong assumptions on the maximal torus and the dominant cocharacter. We will explain in Remark \ref{rem-failure} that these are necessary for our strategy to work. However, we also want to remark here that this problem only appears when working with the untruncated versions. In fact, for $\infty$-truncated Shtukas of Viehmann-Wedhorn, we can show the analogous result in full generality, cf. Section \ref{sec-mot}.
 
\medskip

Beyond this structural statement, our methods provide general tools for working with motives on ind- and pro-Artin stacks.  In particular, we prove: 

\begin{intro-proposition}[\protect{cf. Section \ref{sec-complement}}]
	Let $f\colon X\rightarrow Y$ be a cartesian\footnote{any morphism representable by schemes is cartesian} morphism of ind-schemes (not necessarily of ind-(finite type)) and let $G$ be a pro-algebraic group with split pro-unipotent kernel acting on $X$ and $Y$ such that $f$ is $G$-equivariant. Then there exists a functor $f_{*}\colon \DM(X/G)\to\DM(Y/G)$ preserving compact objects, with the following properties:
	\begin{enumerate}
		\item The functor $f_{*}$ admits a left adjoint $f^{*}$.
		\item If $f$ is ind-proper, then $f_{*}\simeq f_{!}$.
		\item If $f$ is an open immersion, then $f^{!}\simeq f^{*}$.
	\end{enumerate}
	Moreover, one obtains a localization sequence. In case one has also a lower $!$-functor, we also have $!$/$*$-base change and a projection formula. 
\end{intro-proposition}

The techniques developed here provide a general framework for working with motives on ind- and pro-Artin stacks and should be applicable in related contexts, such as the study of Rapoport–Zink spaces or motivic versions of the Satake equivalence.

\subsection*{Acknowledgement}
I want to thank Chirantan Chowdhury and Torsten Wedhorn for various discussions around this topic. Moreover, I want to thank R\i zacan \c{C}ilo\u{g}lu for making me aware of the concept of \textit{quasi-minuscule cocharacter}.

This project was funded by the Deutsche Forschungsgemeinschaft (DFG, German Research Foundation) TRR 326 \textit{Geometry and Arithmetic of Uniformized Structures}, project number 444845124.

\subsection*{Disclosure of AI use} 
The mathematical results of this article were obtained by the author
without artificial intelligence and first appeared in 2025. The major update of the proof of the main theorem is due to the author, was submitted to the journal at the 20th April 2026 and was obtained without the use of any AI.
ChatGPT was consulted to find mistakes and made the author aware of a missing hypothesis (\textit{cartesian}) in the argument of Proposition \ref{prop.cart.edit}. Moreover, ChatGPT assisted in fixing typos and some minor editorial changes.
The author takes full responsibility for the mathematical content,
references, and final presentation.

\section{The stack of local $G$-shtukas}
\label{sec-GSht}

Let us quickly recall the construction of local $G$-shtukas and its truncated counterpart following \cite{VW}. Let us fix a prime $p>0$ and let $k=\FF_{q}$ be a finite field of characteristic $p>0$. Further, let $G$ be a split reductive group scheme over $k[\![z]\!]$. Note that $G$ admits a model over $k$, which we also denote by $G$. Moreover, we fix a split maximal torus $T$ of $G$ and denote the set of dominant cocharacters of $T$ by $X_{*}(T)_{+}$. 

As always one can define the loop group of $G$
$$
	LG\colon R\mapsto G(R(\!(z)\!)),
$$
as well as the positive loop group of $G$
$$
	L^{+}G\colon R\mapsto G(R[\![z]\!])
$$
as fpqc-sheaves on the category of $k$-schemes.
These are representable by ind-schemes and $L^{+}G$ is a subsheaf of $LG$. Note that $LG = \colim (LG)_{i}$ is an ind-affine ind-scheme but not of ind-(finite type).

\begin{defi}[\protect{\cite{HV}}]
	A local $G$-shtuka over a $k$-scheme $S$ is a tuple $(\Ecal,\phi)$, where $\Ecal$ is an $L^{+}G$-torsor on $S$ and $\phi\colon \sigma^{*}\Ecal[1/z]\to \Ecal[1/z]$ is an isomorphism of $LG$-torsors\footnote{Here we $\Ecal[1/z]$ denotes the image of $\Ecal$ under the map $\check{H}^{1}(S,L^{+}G)\to \check{H}^{1}(S,LG)$, cf. \cite{HV}.}.
	
	We denote by $\GSht(S)$ the groupoid of local $G$-shtukas over $S$ and by $\GSht$ the associated fpqc-stack.  
\end{defi}

One can realize $\GSht$ as a ind-pro-Artin stack in the following sense. Let $L^{+}G$ act on $LG$ via Frobenius conjugation, i.e. $\Ad_{\sigma}\colon (a,x) \mapsto ax\sigma(a)^{-1}$. Then we see that $\GSht\simeq LG/\Ad_{\sigma} L^{+}G$ as fpqc quotient stacks. In particular, $\Sht_{G}$ admits a covering by $LG$, an ind-scheme. 

For any dominant cocharacter $\mu\in X_{*}(T)_{+}$, the Bruhat stratification on the affine Grassmannian induces a stratification on $LG$. We denote the associated Schubert cells by $LG_{\mu}$ and the associated Schubert varieties by $LG_{\leq \mu}$. 
Note that $LG_{\leq \mu}\subseteq LG$ is a closed embedding and thus $LG_{\leq \mu}$ is a representable by an affine scheme. Moreover, $LG_{\mu}\subseteq LG_{\leq \mu}$ is affine open embedding. 
This defines a stratification on $\coprod_{\mu\in X_{*}(T)_{+}}LG_{\mu}\to LG$.
We denote the associated quotients 
\[
	\GSht^{\leq \mu}\coloneqq \quot{LG_{\leq \mu}/\Ad_{\sigma}L^{+}G}_{\mathrm{fpqc}},\quad \GSht^{\mu}\coloneqq \quot{LG_{\mu}/\Ad_{\sigma}L^{+}G}_{\mathrm{fpqc}}.
\]

Now let us recall some theory on the moduli of $N$-truncated local $G$-shtukas, for $1\leq N\leq \infty$, following \cite{VW}. We will give a definition of this stack by quotients of truncated loop groups. Let us define those first.

For $1\leq N\leq \infty$, we set
$$
	L^{(N)}G : R\mapsto G(R[\![z]\!]/(z^{N})).
$$
Moreover, for any cocharacter $\mu\in X_{*}(T)_{+}$, we obtain groups $E_{N}(G,\mu)$ acting on $L^{(N)}G$ and transition maps $E_{N+1}(G,\mu)\to E_{N}(G,\mu)$ (cf. \cite{VW}). We will not make the group $E_{N}(G,\mu)$ explicit except for $N=1$ later on, since otherwise we will not need the explicit description of $E_{N}(G,\mu)$.

\begin{defi}[\protect{\cite{VW}}]
	We define the moduli stack of $N$-truncated local $G$-shtukas via 
	$$
		\nSht{N}_{G}^{\mu}\coloneqq \quot{L^{(N)}G/E_{N}(G,\mu)}_{\mathrm{fpqc}}.
	$$
\end{defi}
The composition of the canonical map $L^{(N+1)}G\to L^{(N)}G$ as well as the truncation map $E_{N+1}(G,\mu)\to E_{N}(G,\mu)$ yields a truncation map

\[
	\tau_{N}\colon \nSht{(N+1)}_{G}^{\mu}\to \quot{L^{(N)}G/E_{N+1}(G,\mu)} \to \nSht{N}_{G}^{\mu},
\]
where $E_{N+1}(G,\mu)$ acts on $L^{(N)}G$ via restriction along the truncation map.

\begin{rem}
\label{rem-Sht-coprod}
	Similarly to $\GSht$ one can define $\nSht{N}_{G}$ for any $0\leq N\leq \infty$ equipped with a canonical map $\nSht{\infty}_{G}\to\GSht$, cf. \cite{VW}. However, if $G$ is split, we have an equivalence 
	\[
		\nSht{N}_{G}\simeq \coprod_{\mu\in X_{*}(T)_{+}}\nSht{N}_{G}^{\mu}.
	\]
	In particular, we will not recall the specific definition of $\nSht{N}_{G}$ and define it as the coproduct above.
\end{rem}

In this article, we will be interested in pro-algebraic group schemes arising through successive extensions of unipotent group schemes. Let us define this notion.

\begin{defi}
	Let $G$ be a pro-algebraic group indexed by $\NN$. We say that $G$ has \textit{split pro-unipotent kernel} if $\Ker(G\to G_{1})$ is split pro-unipotent, cf. \cite[App. A]{RS1}.
\end{defi}

\begin{rem}
	Let us  note that in the definition of $\nSht{N}_{G}^{\mu}$ as well as $\GSht$ and $\GSht^{\mu}$, we can equivalently consider the \'etale sheafification since $L^{+}G$ and $E_{\infty}(G,\mu)$ have split pro-unipotent kernel \cite[Prop. A.4.9]{RS1}. In particular, we will only consider \'etale stack quotients from here on.
\end{rem}

\begin{rem}
\label{rem-infty}
	Let us remark that the stacks $\nSht{\infty}_{G}^{\mu}\cong \GSht^{\mu}$ are equivalent through a canonical map $\nSht{\infty}_{G}\to \GSht$ \cite[Thm. 6.21 (2)]{VW}. This is the main technical point, which allows us to compute $M^{c}(\Sht_{G}^{\leq \mu})$ (see Definition \ref{def-motive}) in terms of motives of $\Zip{G}^{\mu'}$ for $\mu'\leq \mu$.
\end{rem}

The stack $\nSht{1}_{G}^{\mu}$ can be described more explicitly in terms of $G$-zips. To define the stack of $G$-zips let us note that the cocharacter $\mu$ yields a parabolic $P_{+}\subseteq G$ with Levi-decomposition $P_{+} = U_{+} \rtimes L$. Let $P_{-}$ denote the opposite parabolic with opposite Levi-decomposition $P_{-} = U_{-}\rtimes L$. Then we let $E_{\mu}$, defined as 
\[
	E_{\mu}\coloneqq \lbrace (lu_{+},l^{(q)}u_{-})\mid l\in L, u_{+}\in U_{+}, u_{-}\in U_{-}^{(q)} \rbrace\subseteq P_{+}\times P_{-}^{(q)},
\]
act on $G$ via $(e,e')\cdot g\coloneqq ege'^{-1}$. The quotient stack $ \Zip{G}^{\mu} = \quot{G/E_{\mu}}$ is the stack of $G$-zips of type $\mu$ of \cite{PWZ}.

In \cite[Cor. 6.27]{VW} they show that there exists a universal homeomorphism 
\[
	\phi\colon \nSht{1}_{G}^{\mu}\to \Zip{G}^{\mu}.
\]
In the following, we will need a slight extension of this result in the sense that the universal homeomorphism $\phi$ is actually compatible with the quotient structure:

\begin{lem}
\label{lem.univ.hom}
	There exists a universal homeomorphism $E_{1}(G,\mu)\to E_{\mu}$ that is compatible with the respective actions on $G$ inducing $\phi$.
\end{lem}
\begin{proof}
	The group $E_{1}(G,\mu)$ is given by $(U_{+}\times U_{-})\rtimes L$ \cite[Prop. 6.13]{VW}.  The action of $E_{1}(G,\mu)$ on $G$ is given by $(e,e')\cdot g \coloneqq eg(e'^{(q)})^{-1}$. 
	Note that $E_{\mu}$ can be rewritten as $(U_{+}\times U_{-}^{(q)})\rtimes L$ where the action of $L$ on $U_{-}^{(q)}$ is induced by the $q$-Frobenius. 
	The action on $G$ is given by $(e,e')\cdot g\coloneqq ege'^{-1}$. Therefore, the relative Frobenius $P_{-}\rightarrow P_{-}^{(q)}$ induces a group scheme homomorphism
	$$
		E_{1}(G,\mu)\rightarrow E_{\mu}
	$$
	that is compatible with the actions on $G$ and is a universal homeomorphism \cite[0CCB]{stacks-project}.
\end{proof}

\section{Motives on pro- and ind-Artin stacks}
In the following we want to understand the motive of $\GSht$ representing Borel-Moore homology. Recall, that each of the $\GSht^{\mu}$ is a quotient of an ind-scheme by a pro-algebraic group scheme. However, as shown in \cite{VW}, the truncated versions $\nSht{N}_{G}^{\mu}$ are in fact Artin stacks. Thus to understand motives on $\nSht{\infty}_{G}^{\mu}$ and therefore on $\Sht_{G}^{\mu}$, we need to understand motives on cofiltered systems of Artin stacks. We call these types of cofiltered system of Artin stacks \textit{pro-algebraic} (we will give a  precise definition below). 

\subsection{Recollection on the theory of motives}

For ind-Artin stacks Richarz-Scholbach defined an $\infty$-category of rational motives via right Kan extension of the classical theory on schemes \cite{RS1}.
In \cite{MotPro} the author has shown that for pro-algebraic stacks, we can define another $\infty$-category rational motives, which we will recall below, via a colimit process. For ind-Artin stacks that also admit a presentation as a pro-algebraic stack this yields a priori two different notions of the $\infty$-category of motives. But, as we will see, in special case like $\nSht{\infty}_{G}^{\mu}$ both definitions of $\DM$ become equivalent and we can benefit from both formalisms.

The benefit of the construction in \cite{MotPro} is that there is, under some assumptions, still a six functor formalism (see Theorem \ref{thm.6.ff}). In \cite{RS1} the existence of a $6$-functor formalism is only shown for ind-schemes of ind-(finite type). However, the insights of \cite{MotPro} enable us to obtain a six functor formalism for rational motives on ind-Artin stacks under the usual finiteness assumptions. In fact, we will see that the existence of a localization sequence can be shown without any finiteness assumption. This is a key insight that enables us to compute motives of stratified stacks (e.g. $\GSht$ which does not admit a covering by ind-schemes of ind-(finite type)) by the motives of the strata. \par

Let us remark here that we do need both formalisms of \cite{RS1} and \cite{MotPro}. The stack $\GSht$ is a priori just a stack for the fpqc-toplogy. However, the explicit construction as a quotient reduces to the study of motives on the ind-scheme $LG$ which is not if ind-(finite type) and motives on $L^{+}G$, which is a pro-scheme. Moreover, the stack $\nSht{\infty}_{G}$, which we will use to compute the motive of $\GSht$, is an honest pro-algebraic stack.

\medskip
In the following we recall some aspects of rational motives needed in this article, following \cite{CD}, \cite{RS1}, and \cite{MotPro}.\par
Let $p>0$ be a prime and let $S$ be a regular scheme of finite type over $\FF_{p}$.
For any scheme $X$ of finite type over $S$, denote by $\DM(X)$ the $\infty$-category of Beilinson motives over $X$ \cite{CD}. The assignment $X\mapsto \DM(X)$, $f\mapsto f^{*}$ satisfies h-descent and forms part of a six-functor formalism, due essentially to Ayoub \cite{AyoubThesis} and in the $\infty$-categorical setting to Robalo \cite{Robalo}. We will not recall details here, as this is classical and beyond the scope of this summary, but we emphasize two relevant facts about $\DM$.

Since $\DM$ satisfies \'etale descent, it extends to Artin stacks via gluing along a smooth atlas. Note, however, that the comparison between motivic cohomology and $K$-theory fails in the equivariant setting, though there remains a comparison with Edidin–Graham’s equivariant Chow groups \cite[Thm.~2.2.10]{RS1}.

Richarz–Scholbach show that this construction yields an $\infty$-category of rational motives for ind-Artin stacks locally of finite type over $S$ \cite{RS1}. In this case one must distinguish between two gluings: using $f^{*}$ or $f^{!}$. For ind-Artin stacks these differ in general. We write $\DM^{*}(X)$ for the category obtained by gluing via $f^{*}$ and $\DM^{!}(X)$ for that obtained via $f^{!}$. For technical reasons, Richarz–Scholbach work with $\DM^{!}$. Moreover, the six functors extend in the setting of ind-schemes under suitable assumptions (cf.~\emph{loc.\ cit.}), and an \'etale realization functor still exists \cite[Thm.~2.3.7]{RS1}.

Coming back to pro-systems of Artin stacks, we will summarize the necessary definitions and the main result of \cite{MotPro} for motives on pro-systems of Artin stacks.

\begin{defi}[\protect{\cite{MotPro}}]
\label{defi.p-stack}
Let $x\colon \NN^{\op}\rightarrow \PreStk_{S}$ be a diagram of prestacks over $S$ and let $X\coloneqq \lim x$ be the associated limit prestack. We call the tuple $(x,X)$ a \textit{pro-algebraic stack}, if $x(i)$ is an algebraic stack locally of finite type over $S$ for all $n\in \NN$.\par 
	If $(x,X)$ is a pro-algebraic stack, then we denote by $X_{n}\coloneqq x(n)$ and the transition maps $X_{m}\rightarrow X_{n}$ by $x_{nm}$, for $n\leq m\in\NN$.\par
	There is an obvious notion of a morphisms and properties.
\end{defi}

\begin{example}
	The assignment $\tau\colon N\mapsto \nSht{N}_{G}$ yields by design a pro-algebraic stack. In \cite{MotPro} the functor $\tau\colon N\mapsto \mathrm{Disp}_{N}$ was considered and motives on $(\tau,\Disp)$ were studied, where $\mathrm{Disp}_{N}$ denotes the stack of $N$-truncated displays of Lau \cite{LauDisp}. 
\end{example}
 
 \begin{defi}[\protect{\cite{MotPro}}]
 	Let $(x,X)$ be a pro-algebraic stack. Then we say that $(x,X)$ is
	\begin{enumerate}
		 \item \textit{classical} if $x_{nm}$ is affine for all $n\leq m\in \NN$,
		\item \textit{tame} if $x_{nm}$ is smooth for all $n\leq m\in \NN$, and
		\item \textit{strict} if $x_{nm}^{*}\colon \DM^{*}(X_{n})\rightarrow \DM^{*}(X_{m})$ is fully faithful for all $n\leq m\in \NN$.
	\end{enumerate}
 \end{defi}

Let us note that the property \textit{classical} yields representability of the limit and strictness ensures that the motivic cohomology of the limit $X$ agrees with the motivic cohomology of $X_{1}$.

\begin{lem}
\label{lem-strict}
	Let $G$ be a split reductive group and $\mu\in X_{*}(T)_{+}$ a cocharacter. The pro-algebraic stack $(\tau, \nSht{\infty}_{G}^{\mu})$ is strict tame and thus $(\tau,\nSht{\infty}_{G})$ is also strict tame.
\end{lem}
\begin{proof}
	Let us first show that $(\tau, \nSht{\infty}_{G}^{\mu})$ is strict tame. The morphism $E_{N+1}(G,\mu)\to E_{N}(G,\mu)$ is surjective with kernel given by a vector group \cite[Prop. 6.16]{VW}. 
	Therefore, by $\AA^{1}$-invariance \cite[Prop. 2.2.11]{RS1}, it is enough to show that for any $N'>N$ the map $L^{(N')}(G)\rightarrow L^{(N)}(G)$ is a torsor under a unipotent group. But this follows from \cite[Prop. A.4.9]{RS1}.
	
	Let us now show how this implies that also $(\tau,\nSht{\infty}_{G})$ is strict tame. By the split assumption on $G$, we have for all $N\in\NN$ a decomposition
	$$
		\nSht{N}_{G} = \coprod_{\mu\in X_{*}(T)_{+}} \nSht{N}_{G}^{\mu}
	$$
	\cite{VW}. This reduces us to the first case, which we already proved. 
\end{proof} 

For a pro-algebraic stack $(x,X)$, we construct a DG-category of motives on $(x,X)$ via the formula
$$
	\DM^{*}(x,X)\coloneqq \colim_{n\in \NN, x^{*}} \DM^{*}(X_{n})
$$
where we take the colimit in $\mathrm{Mod}_{\Dcal(\QQ)}(\Pr^{L})$.
This assignment can be made functorial and in \cite{MotPro} the author shows that we still obtain a $6$ functor formalism, which is motivic in a suitable sense.

\begin{thm}[\protect{\cite[Thm. 5]{MotPro}}]
\label{thm.6.ff}
	Let $(x,X)$ be a pro-algebraic stack. Then $\DM^{*}(x,X)$ is a compactly generated  closed symmetric monoidal $\infty$-category. The assignment 
	$$
	(x,X)\mapsto \DM^{*}(x,X),\quad f\mapsto f^{*}
	$$
	is part of a $6$-functor formalism $(f^{*}\dashv f_{*}, f_{!}\dashv f^{!},\otimes\dashv \Homline)$ on $!$-adjointable\footnote{See Example \ref{ex-adjointable}.} morphisms locally finite type. If one restricts to morphisms, where each square is cartesian, then this $6$-functor formalism is also motivic and orientable in the sense that it admits an orientation satisfies homotopy invariance, $T$-stability, localization, purity and Tate-twists compute Thom-motives.
\end{thm}

\begin{example}
\label{ex-adjointable}
	The condition $!$-adjointable is satisfied in all of the cases that interest us. For example if $(x,X)$ is strict tame and $(y,Y)$ is the constant diagram. In particular, for the purpose of this article the reader may forget about this condition, as it will be satisfied by the objects of interest.
\end{example}

While for ind-Artin stacks, we have the ambiguity that there are two possible definitions of $\DM$ via $f\mapsto f^{*}$ or $f\mapsto f^{!}$, this is not the case for tame pro-algebraic stacks. 

\begin{prop}[\protect{\cite[Prop. 3.13]{MotPro}}]
\label{prop-shriek}
	Let $(x,X)$ be a tame pro-algebraic stack. Then we have an equivalence 
	$$
		\DM^{*}(x,X)\simeq \colim_{n\in \NN, x^{!}} \DM^{!}(X_{n}).
	$$
\end{prop}

In a similar fashion classical pro-algebraic stacks $(x,X)$ have representable limits. This allows a continuity result, which connects the two definitions $\DM^{*}(x,X)$ and $\DM^{*}(X)$.

\begin{prop}[\protect{\cite[Prop. 2.13]{MotPro}}]
	\label{prop-classical}
	Let $(x,X)$ be a classical pro-algebraic stack. Then we have an equivalence 
	\[
		\DM^{*}(x,X) \simeq \DM^{*}(X).
	\]
\end{prop}

Finally, let us define the notion of a motive. We will only need the motive representing Borel-Moore homology.

\begin{defi}
\label{def-motive}
	Let $(x,X)$ be a strict tame pro-algebraic stack over $S$ with structure map $f\colon (x,X)\rightarrow S$, where we see $S$ as a constant diagram. Then we define 
	$$
		M^{c}_{S}(x,X)\coloneqq f_{*}f^{!}1_{S}.
	$$
	If the base is clear, then we may omit the subscript $S$.\par 
	If $X$ is representable by an algebraic stack locally of finite type, we will use the notation $M^{c}_{S}(X)$ to denote the associated motives in $\DM^{*}(S)$.
\end{defi}

Note that $M^{c}_{S}(X)$ represents Borel-Moore homology (or dually cohomology with compact support).

 \subsection{Complements on motives on ind-pro-Artin stacks}
 \label{sec-complement}
 Recall the notation of Section \ref{sec-GSht}.
 In this short subsection, we want to construct for the structure map $p\colon \GSht\to \Spec(k)$ a $!$-pullback and a $*$-adjunction. Moreover, we will need later on a localization sequence, to describe $p_{*}p^{!}1_{k}$ as a decomposition of $M^{c}$ of the strata. Note that we cannot guarantee the existence of $p_{!}$ since the loop group $LG$ is not of ind-(finite type) and so $\GSht$ is not even locally of ind-(finite type) in any suitable sense. This can be avoided once we extend the theory using methods of \cite[App. C]{Hoyois}. However, we will not make use of this extension, since we will not need it for our application.
  
\vspace{6pt} We begin with the following observations. Let $X$ be an (ordinary) Artin stack over $k$, then $\DM^{*}(X)$ is equivalent\footnote{The idea is the same as in \cite{Scholze6ff}. One checks that $(p_{\bullet}^{!})$ admits a left adjoint and the counit is given by $\colim_{[n]\in\Delta} p_{n!}p_{n}^{!}\rightarrow \id$ which by purity and descent for $\DM$ is an equivalence.} to $\DM^{!}(X)$ via $(p_{\bullet}^{!})\colon \DM^{*}(X)\to \lim_{\Delta}\DM^{!}(U^{/X}_{\bullet})$, where $p\colon U\to X$ is a smooth cover and $U^{/X}_{\bullet}$ denotes the \v{C}ech nerve. It is clear from the construction that under $\varphi_{X}\colon \DM^{*}(X)\simeq \DM^{!}(X)$ for any map $f\colon X\rightarrow Y$ locally of finite type, we have 
 \begin{equation}\label{eq.f!}
	 f^{!}\simeq \varphi^{-1}_{Y}f^{*}\varphi_{X}	
 \end{equation}

 Now let us show how to construct the $*$-adjunction for any morphism of ind-schemes. Our methods do not use that we have ind-schemes at all, so instead we will even show the existence of the $*$-adjunction for morphisms of ind-Artin stacks (again without any finiteness assumption).

Let $X=\colim_{i} X_{i}$ be an ind-Artin stack with transition maps $t_{ij}$.  Since each $t_{ij}$ is a closed immersion (\'etale locally), we have $t_{ij*}\simeq t_{ij!}$ and thus an equivalence 
\begin{equation}\label{eq.colim}
	 	\colim_{t_{ij!}}\DM^{!}(X_{i})\simeq \colim_{t_{ij*}}\DM^{*}(X_{i})
\end{equation}

Consider a morphism $f\colon X\to Y$ of ind-Artin stacks (not necessarily locally of ind-(finite type)). We assume that $f$ is \textit{cartesian}, i.e. admits a presentation such that each square
\[
	\begin{tikzcd}
	X_i\arrow[r,"t^{X}_{ij}"]\arrow[d,"f_{i}"]& X_{j}\arrow[d,"f_{j}"]\\
		Y_i\arrow[r,"t^{Y}_{ij}"]&Y_{j}
	\end{tikzcd}
\]
is cartesian, e.g. when $f$ is representable by schemes. By functoriality and (\ref{eq.colim}) we obtain for any $f\colon X\to Y$ a $*$-pushforward $f_{*}\colon \DM^{!}(X)\to \DM^{!}(Y)$, induced by $f_{i*}$ on each level. Moreover, since $t_{ij}$ are closed immersions, the induced square
\[
 	\begin{tikzcd}
		\DM^{*}(X_{i})\arrow[r,"t^{X}_{ij*}"]\arrow[d,"f_{i*}"]& \DM^{*}(X_{j})\arrow[d,"f_{j*}"]\\
		\DM^{*}(Y_{i})\arrow[r,"t^{Y}_{ij*}"]&\DM^{*}(Y_{j})
	\end{tikzcd}
\]
 is left adjointable, i.e. the exchange map 
\[
 	f_{i}^{*}t^{Y}_{ij*}\to t^{X}_{ij*}f_{j}^{*}
\]
 is an equivalence by proper base change. Since also each $t_{i}$ is a colimit preserving functor, we can use\footnote{Note that in \textit{loc. cit.} the colimit along $*$-pullbacks is considered, so one has to be careful with the notation.} \cite[Lem. 3.9]{MotPro} to see that $f_{*}$ admits a left adjoint 
 \[
 f^{*}\colon \colim_{t^{Y}_{ij*}} \DM^{*}(Y_{i})\to \colim_{t^{X}_{ij*}}\DM^{*}(X_{i}).
 \] 
Hence, via (\ref{eq.colim}), we obtain a left adjoint $f^{*}\colon \DM^{!}(Y)\to\DM^{!}(X)$ of $f_{*}$.

 \begin{prop}
 	\label{prop.cart.edit}
 	Let $f\colon X\rightarrow Y$ be a cartesian morphism of ind-Artin stacks. Then there exists a functor $f_{*}\colon \DM^{!}(X)\to\DM^{!}(Y)$ that preserves compact objects with the following properties.
	\begin{enumerate}
		\item The functor $f_{*}$ admits a left adjoint $f^{*}$.
		\item If $f$ is ind-proper, then $f_{*}$ is left adjoint to $f^{!}$.
		\item If $f$ is an open immersion, then $f^{!}\simeq f^{*}$.
	\end{enumerate}
 \end{prop}
 \begin{proof}
 	See discussion above for $1$. The discussion above and (\ref{eq.f!}) yield the remaining points.
\end{proof}
% 
% \begin{lem}
% 	Let $X$ be a ind-Artin stack of ind-finite type over $S$ \hl{NEEDED?}. Further, let $i\colon Z\hookrightarrow X$ be a closed immersion of ind-Artin stacks with open complement $j\colon U\hookrightarrow X$. Then we have the following (co)fiber sequences
%	\begin{align*}
%		i_{!}i^{!}&\to \id \to j_{*}j^{*},\\
%		j_{!}j^{!}&\to \id \to i_{*}i^{*}. 
%	\end{align*}
% \end{lem}
% \begin{proof}
% 	Each of the functors preserve compact objects. So, we can check this on compacts, where we use that any compact of $M\in\DM^{*}(X)$ comes from some finite level $M_{i}\in \DM^{*}(X_{i})$. In particular, the assertion reduces to Artin stacks, where it is known by descent.
% \end{proof}
 
 \begin{rem}
Since every compact on an ind-Artin stack comes from some finite level, we can prove localization sequence for ind-Artin stacks without any finiteness assumption. Under finiteness assumptions or using \cite[App. C]{Hoyois}, we obtain also a lower $!$-functor and can prove the $!$/$*$-exchange equivalence and projection formula similarly. Moreover, analogously to the proof of \cite[Thm. 2.4.2]{RS1} we have an adjunction $\otimes\dashv \Homline$. In fact, by our methods we obtain all the six functors for any morphism of ind-Artin stacks. This was previously not known, as far as the author knows.  However as remarked in \textit{loc. cit.} in general there is no monoidal unit in $\DM^{!}(X)$ for and ind-Artin stack $X$ (in fact not even if $X$ is representable by an ind-scheme). 
 \end{rem}

 Now assume that there exists a pro-algebraic group $G=\lim_{n\in \NN} G_{n}$ over $k$ with split pro-unipotent kernel acting on an ind-scheme $X= \colim_{i}X_{i}$ (not necessarily of ind-(finite-type)) over $k$.  Then let us look at the \'etale stack quotient $p\colon \quot{X/G}\to \Spec(k)$. The quotient map $X\to \quot{X/G}$ is an effective epimorphism in the \'etale topology. Thus, by descent we can compute 
 \[
 \begin{tikzcd}
 	 \DM^{!}(X/G)\simeq \lim (\DM(X)\arrow[r,"",shift right= 0.3em] \arrow[r,"",shift left= 0.3em] &\DM^{!}(X\times G) \arrow[r,"",shift right= 0.5em]\arrow[r,""] \arrow[r,"",shift left= 0.5em]&\dots ).
 \end{tikzcd}
 \] 
 via the limit along the \v{C}ech nerve of $X\to \quot{X/G}$. To construct the adjunction $p^{*}\dashv p_{*}$ and the localization sequence it is enough to glue the functors along the \v{C}ech nerve. We have already discussed that on $\DM(X)$ there is already a $*$-adjunction. So let us handle the $\DM(X\times G^{m})$ for $m\in \NN$. We have 
\begin{equation}
\label{eq.DM!}
	\DM^{!}(X\times G^{m})\simeq \colim_{(t_{ij}\times \id_{G^{m}})_{!}}\DM^{!}(X_{i}\times G^{m})
\end{equation}
noting that $G$ is in fact a faithfully flat group scheme over $k$, so $t_{ij}\times \id_{G^{m}}$ is a closed immersion and its $!$-pushforward makes sense. Since $G$ has split pro-unipotent kernel it is in particular classical and tame by our definition. Therefore, we have 
\begin{equation}
\label{eq.DM*}
	\DM^{!}(X_{i}\times G^{m})\simeq \colim_{n\in \NN,p_{n}^{!}}\DM^{!}(X_{i}\times G^{m}_{n})\simeq \colim_{n\in \NN,p_{n}^{*}}\DM^{*}(X_{i}\times G^{m}_{n})
\end{equation}
by Proposition \ref{prop-shriek} and \ref{prop-classical}, where $p_{n}\colon G^{m}_{n+1}\to G^{m}_{n}$ denotes the canonical projection. Similar to our discussion on the $*$-adjunction on ind-Artin stacks, equivalences (\ref{eq.DM!}) and (\ref{eq.DM*}) yield an adjunction $p^{*}\dashv p_{*}$ and a localization sequence, allowing us to compute $M^{c}(\GSht)$ later on.

\begin{lem}
\label{lem.loc.seq}
	Let $X$ be an ind-scheme over $k$ and $i\colon Y\hookrightarrow X$ a closed immersion with complement $j\colon U\hookrightarrow X$. Moreover, let $G$ be a  pro-algebraic group with split pro-unipotent kernel acting on $X$. Assume further that the action of $G$ restricts to an action on $Y$ and $U$ such that $i$ and $j$ are $G$-equivariant. Let $p^{X}\colon X\rightarrow \Spec(k)$ denote the structure map and similarly for $Y$ and $U$ we denote by $p^{Y}$ and $p^{U}$ the structure maps. Then for $M\in\DM(k)$ the following sequence
	\[
		p^{Y}_{*}p^{Y!}M\to p^{X}_{*}p^{X!}M\to p^{U}_{*}p^{U!}M
	\]
	is a fiber sequence for any $\DM(k)$.
\end{lem}
\begin{proof}
	See discussion above.
\end{proof}

\section{Motives on local $G$-shtukas}
 \label{sec-mot}
 Before continuing with the proof of the main theorem let us a conceptual lemma addressing the potential ambiguity between the definition of $\DM$ for prestacks via Kan extensions and $\DM$ for pro-algebraic stacks via colimits.

\begin{lem}
\label{lem-mot-kan}
	Let $S$ be a scheme.
 	We fix a pro-system $X_{\bullet}\colon \NN^{\op}\to \Sch^{\mathrm{lft}}_{S}$ of locally of finite $S$-schemes. Moreover, we assume that each $X_{m}\to X_{n}$ for $m\geq n$ is an affine bundle. 
	By definition, we set $X_{\infty}\coloneqq \lim_{n}X_{n}$.
	
	We further fix a pro-algebraic group scheme $G_{\infty} = \lim_{n} G_{n}$ with split pro-unipotent kernel such that each $G_{n}$ acts on $X_{n}$ and the transition maps $X_{m}\to X_{n}$ are equivariant with respect to $G_{m}\to G_{n}$. 
	
	The action of $G_{\bullet}$ on $X_{\bullet}$ yields a pro-algebraic stack $(x,\quot{X_{\infty}/G_{\infty}})$.
	
	\medskip
	
	The canonical map $\DM^{*}(x,\quot{X_{\infty}/G_{\infty}})\rightarrow  \DM^{!}(\quot{X_{\infty}/G_{\infty}})$ is an equivalence.
\end{lem}
\begin{proof}
	First note that $\DM^{*}(x,\quot{X_{n}/G_{n}})\simeq \DM^{!}(x,\quot{X_{n}/G_{n}})$ by Proposition \ref{prop-shriek}. The rest of the proof is completely analogous to the proof of \cite[Prop. 4.15]{MotPro}
\end{proof}

 \subsection{Application to the stack of local $G$-shtukas}
 
We will use the notation of Section \ref{sec-GSht}: Let $k=\FF_{q}$ be a finite field of characteristic $p>0$. Let $G$ be a split reductive group scheme over $k[\![z]\!]$ with split maximal torus $T$. 

In this section, we want to show that the compactly supported cohomology of $\GSht$ can be computed by the compactly supported cohomology of $\Zip{G}$, which equivalently is given by the Chow group of $\Zip{G}$ (since $\Zip{G}$ is smooth and $0$-dimensional). 

First, we show the statement for $\nSht{1}_{G}$. Then we will use the presentation of $\nSht{\infty}_{G}^{\mu}$ as a pro-algebraic stack to show that $M^{c}(\nSht{\infty}_{G}^{\mu})\simeq M^{c}(\Zip{G}^{\mu})$. Afterwards, we will compute $M^{c}(\GSht)$ using the stratification $\coprod_{\mu} \GSht^{\mu}\rightarrow \GSht$ using that \'etale locally on $k$, we have $\GSht^{\mu}\simeq \nSht{\infty}_{G}^{\mu}$.

\begin{rem}
	Let us note that $M^{c}(\Zip{G}^{\mu})$ is a Tate motive in $\DM(k)$ \cite[Ex. 5.6]{MotGZip}, i.e. contained in the full stable cocomplete subcategory of $\DM(k)$ generated by Tate twists $1_{k}(n)$. In our specific situation it is known that the subcategory of Tate motives admits a $t$-structure realizing to the perverse $t$-structure on $\ell$-adic cohomology \cite{Levine-t}. 
\end{rem}

\begin{lem}
\label{lem-zip-sht}
	For all $\mu\in X_{*}(T)_{+}$ the $*$-pullback functor 
	$$
		\DM^{*}(\nSht{1}_{G}^{\mu})\rightarrow \DM^{*}(\Zip{G}^{\mu})
	$$ 
	is an equivalence. In particular, there exists a fully faithful functor $\DM^{*}(\Zip{G}^{\mu})\rightarrow \DM^{*}(\tau,\nSht{\infty}_{G}^{\mu})$.
\end{lem}
\begin{proof}
Recall that by Lemma \ref{lem.univ.hom} there exists a universal homeomorphism $\phi\colon \nSht{1}_{G}^{\mu}\to \Zip{G}^{\mu}$ that is compatible with the \v{C}ech nerve of the canonical atlas given by $G$ on both sides.
	 Therefore it induces a map of cosimplicial diagrams
	$$
		\begin{tikzcd}
			 \dots\arrow[r,""]\arrow[d,"",shift left = 0.5em]\arrow[d,"",shift right = 0.5em]\arrow[d,""]&\dots\arrow[d,"",shift left = 0.5em]\arrow[d,"",shift right = 0.5em]\arrow[d,""]\\
			 E_{1}(G,\mu)\times G \arrow[r,""]\arrow[d,"",shift left = 0.3em]\arrow[d,"",shift right = 0.3em]& E_{\mu}\times G\arrow[d,"",shift left = 0.3em]\arrow[d,"",shift right = 0.3em]\\
			 G\arrow[r,"\id_{G}"]& G,
		\end{tikzcd}
	$$
	which on each level is a universal homeomorphism. This induces a map of simplicial diagrams
		$$
		\begin{tikzcd}
			 \dots&\dots\arrow[l,""]\\
			 \DM^{*}(E_{1}(G,\mu)\times G)\arrow[u,"",shift left = 0.5em]\arrow[u,"",shift right = 0.5em]\arrow[u,""]& \DM^{*}(E_{\mu}\times G)\arrow[l,""]\arrow[u,"",shift left = 0.5em]\arrow[u,"",shift right = 0.5em]\arrow[u,""]\\
			 \DM^{*}(G)\arrow[u,"",shift left = 0.3em]\arrow[u,"",shift right = 0.3em]& \DM^{*}(G)\arrow[u,"",shift left = 0.3em]\arrow[u,"",shift right = 0.3em]\arrow[l,""],
		\end{tikzcd}
	$$
	that is an equivalence on each level by qfh-descent \cite[Prop. 3.2.5]{Voev}. Hence, we get an equivalence on the limit 
	$$
		\phi^{*}\colon \DM^{*}(\Zip{G}^{\mu})\rightarrow\DM^{*}(\nSht{1}_{G}^{\mu})
	$$
	\cite[Thm. 2.2.16]{RS1}.
\end{proof}

\begin{lem}
\label{lem-infty-decomp}
	The motive $M^{c}(\tau,\nSht{\infty}_{G}^{\mu})$ is equivalent to $M^{c}(\Zip{G}^{\mu})$.	Moreover, we have 
	$$
		M^{c}(\tau,\nSht{\infty}_{G})\simeq \bigoplus_{\mu\in X_{*}(T)_{+}} M^{c}(\Zip{G}^{\mu}).
	$$	
\end{lem}
\begin{proof}
	By Remark \ref{rem-Sht-coprod}, we have
	$$
		M^{c}(\nSht{1}_{G})\simeq \bigoplus_{\mu\in X_{*}(T)_{+}}M^{c}(\nSht{1}_{G}^{\mu}).
	$$
	Now the result follows from Lemma \ref{lem-strict}, Lemma \ref{lem-mot-kan}, Lemma \ref{lem-zip-sht} and \cite[Lem. 4.4]{MotPro}.
\end{proof}

Before we begin, we note that on $X_{*}(T)_{+}$ there exists a length function $l\colon X_{*}(T)_{+}\rightarrow \ZZ_{\geq 0}$ inducing the dominance order \cite[Ex. 4.2.16]{RS1}.

\begin{thm}
\label{thm.main}
	Let $\mu\in X_{*}(T)_{+}$ be quasi-minuscule and assume that $\rk(T) = 1$, then 
	$$
		M^{c}(\GSht^{\leq\mu}) \simeq \bigoplus_{\mu'\leq \mu} M^{c}(\Zip{G}^{\mu'}) .
	$$
\end{thm}
\begin{proof}
	Let us consider the stratification 
	$$
	\coprod_{\mu\in X_{*}(T)_{+}}\nSht{\infty}_{G}^{\mu}\rightarrow \GSht
	$$
	obtained via the discussion in Section \ref{sec-GSht} and Remark \ref{rem-infty}.
	We have a closed immersion 
	$$
	X\coloneqq\GSht^{<\mu} = \coprod_{\mu'<\mu}\nSht{\infty}_{G}^{\mu'}\subseteq \GSht^{\leq \mu}
	$$
	with open complement $\nSht{\infty}_{G}^{\mu}$. By Lemma \ref{lem.loc.seq}, we have a fiber sequence
	$$
	M^{c}(\coprod_{\mu'<\mu}\nSht{\infty}_{G}^{\mu'})\rightarrow M^{c}(\GSht^{\mu})\rightarrow M^{c}(\nSht{\infty}_{G}^{\mu}). 
	$$
	It suffices to show that the fiber sequence above splits. Thus, we are reduced to show that 
	$$
	\Hom(M^{c}(\nSht{\infty}_{G}^{\mu}), M^{c}(\nSht{\infty}_{G}^{\mu'})[1]) = 0.
	$$
	for all $\mu'<\mu$.  By Lemma \ref{lem-zip-sht} and Lemma \ref{lem-mot-kan} together with \cite[Thm. 6.23]{VW}, we have
	$$
	\Hom(M^{c}(\nSht{\infty}_{G}^{\mu}), M^{c}(\nSht{\infty}_{G}^{\mu'})[1]) = \Hom(M^{c}(\Zip{G}^{\mu}), M^{c}(\Zip{G}^{\mu'})[1]).
	$$
	
	We will prove that the right hand side is in fact equal to $0$. 
	
	By design $\mu'$ has length $0$ and thus $L_\mu' = G$ and $L_\mu = T$. Moreover, we have $W_{\mu'} =W$ the Weyl group of $G$ with respect to $T$. Let us remark that $M^{c}(\Zip{G}^{\mu})\simeq R\Gamma(\quot{G/_{\varphi}T},\QQ)$, since $\Zip{G}^{\mu'}$ is a smooth Artin stack of relative dimension $0$. Here, $T$ acts via $\varphi$-conjugation, where $\varphi$ denotes the relative Frobenius. Using comutativity of $\Hom$ with limits, it sufficed to show that
	\[
	\Hom(R\Gamma(\quot{G/_{\varphi}T},\QQ), R\Gamma(\Zip{G}^{\mu'},\QQ)[1]) = 0
	\]
	Again, by twisting and shifting appropriately, we may work $M^c$ instead of $R\Gamma(-,\QQ)$ and show that
	\[
	\Hom(M^c(\quot{G/_{\varphi}T})(-d)[-2d], R\Gamma(\Zip{G}^{\mu'})[1]) = 0,
	\]
	where $d = \dim(G)-1>0$. 
	Let $B$ be a Borel containing $T$. Note that $M^c(T\bs G/T) = M^c(T\bs G/B)(u)[2u]$, where $u = \dim R_u(B)$. Then, we have a fiber sequence of motives
	\[
	M^c(T\bs G/B)(u)[2u] \to M^c(T\bs G/B)(u+1)[2u+2]\to M^c(\quot{G/_{\varphi}T}),
	\]
	\cite[\S 4.1]{MotGZip}.
	In particular, we obtain a morphism 
	\[
	\Homline(M^c(T\bs G/B)(u)[2u+1], R\Gamma(\Zip{G}^{\mu'})(d)[2d+1])\to \Homline(M^c(\quot{G/_{\varphi}T}), R\Gamma(\Zip{G}^{\mu'})(d)[2d+1])
	\]
	with cofiber
	\[
	\Homline(M^c(T\bs G/B), R\Gamma(\Zip{G}^{\mu'})(d-u-1)[2d-2u-1]).	
	\]
	After applying $\pi_0$ to the cofiber, we see that the group vanishes by degree reasons as $2d-2u-2-2d+2u+1 = -1$. Therefore, it is enough to prove 
	\[
		\Hom(M^c(T\bs G/B)(u)[2u+1], R\Gamma(\Zip{G}^{\mu'})(d)[2d])=0.
	\]
	By Poincar\'e duality, we have $M^c(T\bs G/B)\simeq M(T\bs G/B)^\vee$ since $\dim(T\bs G/B) =0$ \cite[Prop. 2.36]{HosLeh}. Moreover, we have 
	\[
	M(T\bs G/B)\simeq \bigoplus_{W}M(BT)(l(w))[2l(w)]\simeq \bigoplus_W\bigoplus_{i\geq 0} \QQ(i+l(w))[2i+2l(w)]
	\]
	\cite[Thm. 1]{MotT}.
	Together with $d-u = \dim(G/B) =1$, we obtain
	\begin{align*}
		\Hom(M^c(T\bs G/B), R\Gamma(\Zip{G}^{\mu'})(1)[2])&= \bigoplus_W\bigoplus_{i\geq 0}\Hom(\QQ, R\Gamma(\Zip{G}^{\mu'})(1+i+l(w))[2+2i+2l(w)])\\ &= \bigoplus_W\bigoplus_{i\geq 1}\mathrm{CH}^{i+l(w)}(\Zip{G}^{\mu'}) \\&= \bigoplus_W(S^W/S^W_+)_{\geq 1+l(w)},
	\end{align*}
	where $S\coloneqq\Sym_\QQ(X^*(T))$ and $S_+$ the elements of degree $> 0$, cf. \cite[\S 4.1]{WZ} for the last equation. In particular, $\Hom(M^c(T\bs G/B), R\Gamma(\Zip{G}^{\mu'})(d)[2d])=0$, completing the proof.
\end{proof}

\begin{rem}
	\label{rem-failure}
	The assumptions in Theorem \ref{thm.main} are crucial for the proof to work. For example, in the $\SL_2$ case take $l(\mu) = 4$ and $l(\mu') = 2$. Then $\mu'<\mu$ but $L_\mu=L_{\mu'} =T$. In particular, the same argument as before shows that 
	\[
		\Hom(R\Gamma(\Zip{G}^{\mu}), R\Gamma(\Zip{G}^{\mu'})(d)[2d]) = \bigoplus_{W} (S/S^W_+)_{\geq 1+l(w)}.
	\]
	This group does not vanish in general, so our arguments of the proof fail. 
	
	Now assume that $\rk(T) = r>1$, $G\neq T$ and choose $\mu>\mu'$ such that $L_\mu = T$. We claim that there are non-trivial morphisms $M^c(\Zip{G}^{\mu})\to M^c(\Zip{G}^{\mu'})[1]$.  Again, we have $M^c(\Zip{G}^{\mu})\simeq R\Gamma(G/_{\varphi}T,\QQ)\simeq M^c(G/_{\varphi}T)(-d)[-2d],$ where $d=\dim(G)-r$. Moreover, we obtain an opposite filtration 
	\[
		M^c(G/_{\varphi}T)\to M_{1}\to\dots M_r\coloneqq M^c(T\bs G/T)[r],
	\]
	where each fiber $M_i\to M_{i+1}$ is equivalent to $M_{i+1}(1)[2+i]$. This yields a filtration $F^\bullet$ of the form
	\[
		\Homline(M^c(T\bs G/T), R\Gamma(\Zip{G}^{\mu'})(d)[2d-r])\to \dots\to \Homline(M^c(G/_{\varphi }T), R\Gamma(\Zip{G}^{\mu'})(d)[2d]).
	\]
	Hence, we obtain a spectral sequence 
	\[
		E_1^{p,q} = \pi_{p+q}\gr_p F^\bullet \Rightarrow \pi_{p+q}\Homline(M^c(G/_{\varphi }T), R\Gamma(\Zip{G}^{\mu'})(d)[2d])
	\]
	\cite[Prop. 1.2.2.14]{HA}.
	Note that for $-r<p\leq 0$, we have $E_1^{p,-p} = \Hom(M_{-p+1},R\Gamma(\Zip{G}^{\mu'})(d-1)[2d-1+p])$ and $E_1^{-r,r} = \Hom(M^c(T\bs G/T), R\Gamma(\Zip{G}^{\mu'})(d)[2d-r]).$ Replacing $M^c(T\bs G/T)$ with $M^c(T\bs G/B)(u)[2u]$, we see that the last term is equivalent to 
	\[
		\Hom(M^c(T\bs G/B), R\Gamma(\Zip{G}^{\mu'})(g)[2g-r]) = \bigoplus_{W}\bigoplus_{i\geq g+l(w)} \mathrm{CH}^{i}(\Zip{G}^{\mu'},r), 
	\]
	where $g = \dim(G/B)$. Analogously, we have 
	\[
		E_1^{-r+1,r-1} = \Hom(M^c(T\bs G/B), R\Gamma(\Zip{G}^{\mu'})(g-1)[2g-2]) = \bigoplus_{W}\bigoplus_{i\geq g-1+l(w)} \mathrm{CH}^{i}(\Zip{G}^{\mu'}).
	\]
	If $l(\mu')\neq 0$, then the latter group does not vanish in general. We suspect that if $l(\mu') = 0$, then all of the $E_1$-terms vanish. However, the first obstruction is 
	\[
		\bigoplus_{i\geq g} \mathrm{CH}^{i+l(w)}(\Zip{G}^{\mu'},r),
	\]
	which is possible to compute but requires a more algorithmic approach that was not pursued by the author.
\end{rem}

\begin{rem}
	Note that if $G$ is non-split, it is still quasi-split. We can use all the constructions of \cite{VW} but we have to be careful with the stratifications by dominant cocharacters. A priori we have to work as in \textit{op.\! cit.} with conjugacy classes of cocharacters and obtain a stratification of $\GSht$ by $\nSht{\infty}_{G}^{\mu}$ only after passage to some split extension of $k$. Thus, the main obstruction existence of a stratification $\coprod \GSht^{\mu}\to \GSht$ which to our knowledge requires the existence of a split maximal torus. 
\end{rem}

%------------------------------------------------------------------

%------------------------------------------------------------------

%==================================================================
\addcontentsline{toc}{section}{References}
\bibliographystyle{plainurl}
\bibliography{MotOnSht}

\end{document}